\documentclass{amsart}
\usepackage[english]{babel}
\usepackage{amsmath,amssymb,dsfont,mathrsfs,anysize,fancyhdr,epsfig}
\usepackage{cite}
\usepackage{graphicx}
\usepackage{url}

\usepackage{tikz}
\usepackage{enumerate}

\usepackage{xcolor}

\newcommand{\NN}{\mathbb{N}}

\newcommand{\QQ}{\mathbb{Q}}
\newcommand{\RR}{\mathbb{R}}

\DeclareMathOperator{\Diam}{Diam}

\newtheorem{verm}{Vermoeden}[section]
\newtheorem{Thm}[verm]{Theorem}

\newtheorem{defi}[verm]{Definition}
\newtheorem{lem}[verm]{Lemma}

\title[Full proof of Kwapie\'{n}'s theorem]{Full proof of Kwapie\'{n}'s theorem on representing bounded mean zero functions on $[0,1]$}
\author{A. F. Ber, M. J. Borst, F. A. Sukochev}

\begin{document}
	
	\date{\today}
	\begin{abstract}
		In \cite{Kwapien}, Kwapie\'{n} announced that every mean zero function $f\in L_\infty[0,1]$ can be written as a coboundary $f = g\circ T -g$ for some $g\in L_\infty[0,1]$ and some measure preserving transformation $T$ of $[0,1]$.
		Whereas the original proof in \cite{Kwapien} holds for continuous functions, there is a serious gap in the proof for functions with discontinuities.  In this article we fill in this gap and establish Kwapie\'{n}'s result in full generality. Our method also allows to improve the original result by showing that for any given $\epsilon>0$ the function $g$ can be chosen to satisfy a bound $\|g\|_\infty\leq (1+\epsilon)\|f\|_\infty$.
	\end{abstract}
	\maketitle

\section*{Introduction}
In this article we prove the following strengthening of theorem announced by Kwapie\'{n} in \cite{Kwapien}.
\begin{Thm}\label{Main result} Let $f\in L_\infty[0,1]$ be a real-valued mean zero function. Choose $\epsilon>0$, then there exists a $g\in L_\infty[0,1]$ with $\|g\|_\infty\leq (1+\epsilon)\|f\|_\infty$ and a measure preserving transformation mod0 $T$ of $[0,1]$ such that $f = g\circ T - g$.
\end{Thm}

The proof given in \cite{Kwapien} is incomplete for non-continuous functions. On the other hand, in the last 20 years, Kwapie\'{n}'s Theorem \ref{Main result} has been used in the theories of symmetric functionals (see e.g. \cite{FK}) and singular traces (see e.g. \cite{SingularTraces}) and featured in some measure theory treaties (see \cite[p.335. Exercise 9.12.68]{Bogachev2}). It is therefore important to obtain its full proof and this is the main objective of the present article. Our proof complements some deep ideas and interesting technical approaches from \cite{Kwapien}.

Now, we briefly explain the nature of the gap in \cite{Kwapien}. The original proof suggested in \cite{Kwapien} is based on the  usage of Lusin's theorem (see Theorem \ref{Prelim:lusin} below), which guarantees the existence of disjoint sets $A_n\subseteq [0,1]$ for $n\geq 1$ such that $\lambda([0,1]\setminus\bigcup_{n=1}^\infty A_n) = 0$ and so that the following holds:
\begin{enumerate}
	\item $A_n$ is a closed subset, homeomorphic to the Cantor set.
	\item $f$ restricted to $A_n$ is a continuous function.
	\item $\lambda(A_n)>0$ and $\int_{A_n}fd\lambda = 0$, where $\lambda$ is Lebesgue measure.
\end{enumerate}
It is further stated in \cite{Kwapien} that for each $n\geq 1$, there exists a homeomorphism from $A_n$ to the Cantor set $\{0,1\}^\NN$ that maps the measure $\frac{\lambda}{\lambda(A_n)}$ to the Cantor measure $\mu$ (the product measure $\mu = \prod_{i=1}^\infty \mu_i$ where $\mu_i$ is the probability measure on $\{0,1\}$ given by $\mu_i(\{0\})=
\mu_i(\{1\}) =\frac12$). 

We shall now present a concrete counter-example to this claim.

Fix an irrational scalar $\alpha\in (0,1)$ and set 
$f = (1-\alpha)\chi_{[0,\alpha)} -\alpha \chi_{(\alpha,1]}$, so $\int_{[0,1]}fd\lambda = 0$.
Now, let $A\subseteq [0,1]$ be  a set such that conditions (1), (2) and (3) are satisfied for this set.
Since $f|_{A}$ must be continuous and $A$ must be compact, it follows that either the set $(\alpha,\alpha+\epsilon)$ or else the set $(\alpha,\alpha - \epsilon)$ is disjoint with $A$ for some $\epsilon>0$.
Suppose, for definiteness, that $(\alpha,\alpha+\epsilon)\cap A=\emptyset$ for some $\epsilon>0$ (the argument for the other case is the same).
We set $C_1 = A\cap [0,\alpha]$ and $C_2 = A\cap [\alpha+\epsilon,1]$
so that we have
$(1-\alpha)\lambda(C_1) -\alpha\lambda(C_2) = \int_{A}fd\lambda = 0$.
This means that $\lambda(C_2) = \frac{1-\alpha}{\alpha}\lambda(C_1)$
and hence $\frac{\lambda(C_1)}{\lambda(A)} = \alpha$ is irrational.
Now, let $\varphi:A\to \{0,1\}^\NN$ be a homeomorpishm.
Since the $C_1$ is open and closed in $A$, it follows that 
$\varphi(C_1)$ is open and compact. 
Hence we can write  $\varphi(C_1) = \bigcup_{U\in \mathcal{A}}U$ for some subset $\mathcal{A}$ 
of the basis $\mathcal{B} = \{\{x\in \{0,1\}^\NN: (x_i)_{i=1}^N = 
(a_i)_{i=1}^N\}: (a_i)_{i=1}^N\in \{0,1\}^N 
\text{ for some } N\in \NN\}$ of the topology of 
$\{0,1\}^\NN$. Now, thanks to the compactness assumption, we can write 
$\varphi(C_1) = \bigcup_{i=1}^M U_i$ for some 
$M\in \NN$ and some open sets 
$U_i\in \mathcal{A}$ for $i=1,..,M$. Next, applying further subdivision, we obtain 
$\varphi(C_1) = \bigcup_{i=1}^{M'}B_i$ for some $M'\in\NN$ and some pairwise disjoint open sets $B_i\in \mathcal{B}$ for $i=1,..,M'$. Such a decomposition implies that 
$\mu(\varphi(C_1)) = \frac{j}{2^l}$ for some 
$j,l\in \NN$ so that $\varphi$ does not map $\frac{\lambda}{\lambda(A)}$ to $\mu$. This contradiction explains that Kwapie\'{n}'s construction does not hold in this particular case.

We give an outline of its structure.
In Section \ref{Section:Prelim}, we set notation and state  some known results that will be used throughout this article.
In Section \ref{Section:ProofNowhereConstantFunctions}, we treat some special class of measurable bounded functions $f$ and  adapt the proof in \cite{Kwapien} to this class. More precisely, we request that for a bounded mean zero function $f$, there exists a 
$\kappa\in \RR$ so that $\lambda(f^{-1}(\{y\}))  = 0$ for all $y<\kappa$ and such that $f|_{\{f\geq \kappa\}} = \kappa$. In that section, we present a modified method of constructing the sets $A_n$ used in \cite{Kwapien} which allows us to treat this subclass.  Next, in Section \ref{Section:ProofCountablyValuedFunctions},
we prove Kwapie\'{n}'s Theorem for countably valued, mean zero functions $f\in L_\infty[0,1]$. Finally, in Section \ref{Section:ProofFullTheorem}, we combine the results from Section \ref{Section:ProofNowhereConstantFunctions} and \ref{Section:ProofCountablyValuedFunctions} and obtain the complete proof of Kwapie\'{n}'s Theorem \ref{Main result}.\\

We note that in \cite{AdamsPublished} (see also \cite{Adams2}), a special case of Kwapie\'{n}'s Theorem is established by using a completely different approach and for selected class of measure preserving transformation.

\section{Preliminaries}
\label{Section:Prelim}
In this section we set notation and state known results that we need. Throughout this article we will equip a Lebesgue measurable  set $K\subset [0,1]$ with the Borel $\sigma$-algebra $\mathcal{B}(K)$ and the (induced) Lebesgue measure $\lambda$, unless otherwise stated. We then let $L_\infty(K)$ denote the space of essentially bounded real-valued functions under the equivalence relation of equality almost everywhere. We denote the essential supremum by $\|\cdot\|_\infty$. \\

We will use the following notion of a measure preserving transformation.
\begin{defi}
	Let $(\Omega,\mathcal{A},\mu)$ and $(\Omega',\mathcal{A}',\mu')$ be measure spaces. We define a measure preserving transformation mod0 between the measure spaces as a bijection $T:\Omega\setminus N\to \Omega'\setminus N'$ for some null sets $N\in \mathcal{A}$ and $N'\in \mathcal{A}'$, so that both $T$ and $T^{-1}$ are measurable mappings and $\mu'(T(A)) = \mu(A)$ for all $A\subseteq \Omega\setminus N$ in $\mathcal{A}$.
\end{defi}

The following result is obtained by a combination of Theorem 7.4.3 and Proposition 1.4.1 in \cite{MT:book1}

\begin{Thm}[Lusin's theorem]\label{Prelim:lusin}
	Let $D\subseteq [0,1]$ be Borel-measurable and let $f: D \to \RR$ be Borel-measurable. If $\epsilon>0$, then there is a compact subset $K\subseteq A$ such that $\mu(A \setminus K) < \epsilon$ and such that the restriction of $f$ to $K$ is continuous.		
\end{Thm}

The following theorem is obtained by combining Theorems 9.3.4 and 9.5.1 in \cite{Bogachev}
\begin{Thm}\label{Prelim:isomorphism}
	Let $A,B\subseteq [0,1]$ be some subsets of equal positive measure, then there exists a measure preserving transformation mod0 $T$ between $A$ and $B$.
\end{Thm}


We shall also need the following two lemmas, which are taken from \cite[Lemma]{Kwapien} and \cite[Lemma 5.2.3]{SingularTraces} respectively. For convenience of the reader, we give a short proof of the second lemma.
\begin{lem}\label{Prelim:permutations-kwapien}
	Let $(a_{i,j})_{n \times m}$ be a matrix of real numbers such that $|a_{i,j}|\leq C$ for $i=1,...,n$ and $j=1,...,m$ and such that $\sum_{j=1}^{m}a_{i,j} = 0$ for $i=1,...,n$. Then there exists permutations $\sigma_1,....,\sigma_n$ of $\{1,...,m\}$ such that 
	\[\left|\sum_{i=1}^{k}a_{i,\sigma_i(j)}\right|\leq 2C \text{ for all } k=1,...,n \text{ and } j=1,....,m.\]
\end{lem}

\begin{lem}
	\label{sum}   Let
	$a_1,\dots,a_n\in\mathbb{R}$ with $\sum_{k=1}^n a_k=0$. Then
	there exists a rearrangement $\sigma$ of the indices $1,\dots,n$, so that
	$|\sum_{k=1}^m  a_{\sigma(k)}|\leq \max_{k=1}^n|a_k|$ for every
	$m\in\{1,\dots,n\}$.
\end{lem}
\begin{proof}[Proof of Lemma \ref{sum}]
	Set $\sigma(1)=1$ and assume that the sought for rearrangement is already defined for $1,\dots,m$, so that $|\sum_{k=1}^m a_{\sigma(k)}|\leq
	\max_{k=1}^n|a_k|$. Since $\sum_{k=1}^n a_k=0$, it follows that among
	the remaining indices (excepting $\sigma(1),\dots,\sigma(m)$) there exists
	an index $j$, such that $\sum_{k=1}^m  a_{\sigma(k)}$ and $a_j$
	have opposite signs. Set $\sigma(m+1)=j$. Then, it follows from the obvious implication $a\leq 0\leq b \Rightarrow a\leq a+b \leq b$ that
	$-\max_{k=1}^n|a_k|\leq \min(\sum_{k=1}^m
	a_{\sigma(k)},a_{\sigma(m+1)}) \leq \sum_{k=1}^{m+1}
	a_{\sigma(k)}\leq \max(\sum_{k=1}^m a_{\sigma(k)},a_{\sigma(m+1)})
	\leq \max_{k=1}^n|a_k|$.
\end{proof}

\section{Theorem \ref{Main result} for almost nowhere constant functions}
\label{Section:ProofNowhereConstantFunctions}
 
In this section we shall always assume 

\begin{enumerate}\label{conditions}
\item[(i)] that $K\subseteq [0,1]$ is a Lebesgue measurable set   with $\lambda(K)>0$ and consider the measure space $(K, \lambda)$ equipped with Lebesgue measure; 

\item[(ii)] that  $f\in L_\infty(K)$ is a real valued mean zero function such that there exists a constant $\kappa \in \RR$ with $\lambda(f^{-1}(\{y\})) = 0$ for $y< \kappa$ and such that $f|_{\{f\geq \kappa\}} = \kappa$. 

\end{enumerate}

In this section, we will prove the following special case of Theorem \ref{Main result}.
\begin{Thm}\label{TheoremNowhereConstant}
Suppose that $K$ and $f$ satisfy assumptions (i) and (ii) above.
For every $\epsilon>0$, there is a $g\in L_\infty(K)$ with 
$\|g\|_\infty\leq (1 + \epsilon)\|f\|_\infty$ 
and a measure preserving transformation $T$  mod0 of 
$(K, \lambda)$ such that $f = g\circ T - g$.
\end{Thm}

We prove Theorem \ref{TheoremNowhereConstant} 
by adapting and adjusting the method suggested 
by Kwapie\'{n} in \cite{Kwapien}. We modify the construction of sets $A_n, n\ge 1$ used in  \cite{Kwapien} in the proof for continuous functions to make them suitable for functions $f$ as above.

\subsection{Preserving mean zero condition for subsets}

For the proof of Theorem \ref{TheoremNowhereConstant} we will need the following result which we briefly explain below. Suppose that we are given a measurable set $D$ and an essentially bounded, non-zero function $f$ supported on $D$, whose integral over $D$ equals $0$. Then, for sufficiently small $\epsilon>0$ (depending on $f$), and any compact subset $E\subseteq D$ with $\lambda(D\setminus E)<\epsilon$ and such that $f|_E$ is continuous, there exists a slightly smaller compact subset $K\subseteq E$ such that $\int_{K}fd\lambda = 0$.
\begin{lem}\label{Lemma:compensation-of-integral}
	Let $D\subseteq [0,1]$ and let $f\in L_\infty(D)$ be a real-valued function satisfying $f\neq 0$ and  $\int_{D}fd\lambda =0$. Then $\tau^+ = \lambda(\{f>\frac{1}{2}||f^+||_\infty\})>0$ and $\tau^- = \lambda(\{f<-\frac{1}{2}||f^-||_\infty\})>0$. Let $$0<\epsilon < \frac{1}{4}\min\{\tau^+\frac{||f^+||_\infty}{||f||_\infty},\tau^-\frac{||f^-||_\infty}{||f||_\infty}\}.$$ Then, for any compact subset $E\subset D$ with $\lambda(E)\geq \lambda(D) - \epsilon$ and for which $f|_E$ is continuous, there is a compact subset $K\subseteq E$ with 
	$\int_{K}fd\lambda = 0$ and $\lambda(K) \geq \lambda(D) -\left(1 + 2||f||_\infty\max\{\frac{1}{||f^+||_\infty},\frac{1}{||f^-||_\infty}\}\right)\epsilon$.
	\begin{proof}
		Let $D,E$ and $f$ and $\epsilon$ be such that all conditions above are fulfilled.
		Further, set $\tilde{\epsilon} = \int_{E}fd\lambda$. Since, $f\not=0$ and $\int_{D}fd\lambda = 0$ we have that $f^+,f^-\not=0$, hence $\tau^+,\tau^->0$. We have $$|\tilde{\epsilon}| = |\int_{E}fd\lambda| =|\int_{D\setminus E}fd\lambda| \leq \lambda(D\setminus E)||f||_\infty \leq \epsilon||f||_\infty.$$ We will further suppose that $\tilde{\epsilon}\geq 0$. The set $f|_E^{-1}((\frac{1}{2}||f^+||_\infty,\infty))$ is open in $E$ since $f|_E$ is continuous. Now, for $0\leq r\leq 1$ define the 
		measurable set, $R_r :=f|_E^{-1}((\frac{1}{2}||f^+||_\infty,\infty))\cap [0,r)\subseteq E$ and let $F:[0,1]\to \RR$ be given by
		$F(r) = \int_{E\setminus R_r}fd\lambda$. Since $|F(r_1)-F(r_2)|\leq \|f\|_\infty|r_1-r_2|$, it follows that $F$ is a continuous function. Further we have
		$F(0) = \tilde{\epsilon}\geq0$ by assumption. Furthermore, since 
		$$\lambda(R_1)=\lambda(f^{-1}((\frac{1}{2}||f^+||_\infty,\infty))\cap E) \geq \tau^+ - \lambda(D\setminus E) \geq \tau^+ - \epsilon\geq \frac{1}{2}\tau^+$$ it follows that $$F(1) = \int_{E}fd\lambda - \int_{R_1}fd\lambda \leq \tilde{\epsilon} - \lambda(R_1)\frac{1}{2}||f^+||_\infty \leq \tilde{\epsilon} - \tau^+\frac{1}{4}||f^+||_\infty\leq(\epsilon - \frac{1}{4}\tau^+\frac{||f^+||_\infty}{||f||_\infty})||f||_\infty<0.$$
We conclude that there exists $r_0\in [0,1]$ with $F(r_0) = 0$. Now set $K = E\setminus R_{r_0}\subseteq D$. Since $R_{r_0}$ is open in $E$, it follows that $K$ is compact. Further, we have $\int_{K}fd\lambda = F(r_0) = 0$. Finally, we have
		$$\tilde{\epsilon} = \int_{E}fd\lambda = \int_{R_{r_0}}fd\lambda \geq \lambda(R_{r_0})\frac{1}{2}||f^+||_\infty,$$ which immediately yields $$\lambda(R_{r_0})\leq \frac{2\tilde{\epsilon}}{||f^+||_\infty} \leq \frac{2||f||_\infty}{||f^+||_\infty}\epsilon.$$
		
Hence,

\begin{align*}
\lambda(K)& = \lambda(E) - \lambda(R_{r_0}) \geq (\lambda(D) -\epsilon) - \frac{2||f||_\infty}{||f^+||_\infty}\epsilon 
 = \lambda(D) - (1+\frac{2||f||_\infty}{||f^+||_\infty})\epsilon \cr &
 \geq \lambda(D) - \left(1 + 2||f||_\infty\max\{\frac{1}{||f^+||_\infty},\frac{1}{||f^-||_\infty}\}\right)\epsilon.
\end{align*}  
The case that $\tilde{\epsilon}<0$ now follows by replacing $f$ by $-f$. This completes the proof.
	\end{proof}
\end{lem}
The next lemma shows that if we have a compact set $K$ and a continuous, function $f$ on $K$ satisfying assumptions (i) and (ii) of Theorem \ref{TheoremNowhereConstant}, then for 
$0< c<\lambda(K)$ we can find a compact subset $E$ of $K$ such that $\lambda(K\setminus E) = c$,
 $\int_{E}fd\lambda = 0$, and such that $E$ does not contain the endpoints of $K$.\\

\begin{lem}\label{Lemma:shrinking}
	Let $K\subset [0,1]$ be a compact measurable set with $\lambda(K)>0$ and let $f\in L_\infty(K)$ be a continuous, real-valued, mean zero function satisfying assumptions (i) and (ii) of Theorem \ref{TheoremNowhereConstant}. For every $c\in (0,\lambda(K))$, there is a compact subset $E\subseteq K\cap (\inf K,\sup K)$ of measure $\lambda(E) = \lambda(K) - c$ such that $\int_E f d\lambda = 0$.
\end{lem}
\begin{proof}
	Let $K$, $f$ and $c$ be given. If $f = 0$ we can simply take $E = K\cap [\inf K + r,\sup K -r]$ for some $r>0$ such that $\lambda(E) = \lambda(K) - c$. Hence, without loss of generality, we shall assume below that $\lambda(f^{-1}(\{0\})) = 0$. Setting 
	$$\upsilon^\pm = \lambda(\{f^\pm>0\})$$ we observe that $\upsilon^\pm>0$ since $f$ is mean zero and does not identically vanish. Now for $0\leq r\leq \upsilon^+$, we set $$A_r = \{f>0\} \cap ([0,\inf K + a_r)\cup (\sup K - a_r,1]),$$ where $a_r\geq 0$ is chosen to satisfy the condition $\lambda(A_r) = r$.
Furthermore, for $0\leq r\leq \upsilon^-$, let 
$$
B_r = \{f<0\} \cap ([0,\inf K + b_r)\cup (\sup K - b_r,1]),
$$ 
where $b_r\geq 0$ is chosen to satisfy the condition $\lambda(B_r) = r$. We have $A_{r_1}\subset A_{r_2}$ and $B_{r_1}\subset B_{r_2}$ whenever $r_1<r_2$.
	Let $F_\pm:[0,\upsilon^\pm]\to \RR^+$ be given by 
$$F_+(r) = \int_{A_r}fd\lambda\quad{\rm
	and}\quad F_-(r) = -\int_{B_r}fd\lambda,
$$
respectively. The functions $F_\pm$ are continuous and strictly increasing and we also have 
$$F_+(0) = F_-(0) = 0\quad{\rm and}\quad 
F_+(\upsilon^+) = F_-(\upsilon^-)$$
 as $f$ is mean zero on $K$. 
	
	Let $G: [0,\upsilon^+]\times [0,\upsilon^-] \to \RR$ be given by
	$G(t,r) = F_-(r) - F_+(t)$ so that we have 
	$$G(t,0) = F_-(0) - F_+(t)\leq 0$$ and
	$$G(t,\upsilon^-) = F_-(\upsilon^-) - F_+(t) = F_+(\upsilon^+) - F_+(t) \geq 0$$ for all $t\in [0,\upsilon^+]$. The function $G(t,\cdot)$ is continuous, and therefore there is a $0\leq x\leq\upsilon^-$ with $G(t,x) = 0$. Further, since $G(t,\cdot)$ is strictly increasing the value of $x$ is uniquely determined. Now we can define $$H:[0,\upsilon^+]\to [0,\upsilon^-]$$ by letting $H(t)\in [0,\upsilon^-]$ be the unique real number satisfying  $G(t,H(t)) = 0$.
	
For $t\in [0,\upsilon^+]$ we now have $F_-(H(t)) - F_+(t) = 0$, or equivalently,
$$
\int_{A_t}fd\lambda + \int_{B_{H(t)}}fd\lambda = 0.
$$
Hence, setting $$E_t = K\setminus (A_t\cup B_{H(t)}),$$ we obtain a compact set (due to the fact that $A_t$ and $B_{H(t)}$ are open in $K$), such that 
$$
\int_{E_t}fd\lambda = 0,\quad \forall t\in [0,\upsilon^+].$$ 
Observing that $A_t\cup B_{H(t)} \subset K$, we have  $$\lambda(E_t) = \lambda(K) - \lambda(A_t\cup B_{H(t)})=\lambda(K) - (\lambda(A_t)+\lambda( B_{H(t)}))= \lambda(K) - (t+H(t)).$$

The function $G$ is continuous, strictly decreasing in $t$ and strictly increasing in $r$, and therefore we have that $H$ is continuous and strictly increasing. Observing that $H(0) = 0$ and that $$\upsilon^+ + H(\upsilon^+) = \upsilon^+ + \upsilon^- = \lambda(K)
$$ 
(the latter follows from the fact that $f$ is non-zero almost everywhere on $K$), 
we infer that there exists a real number $t_0\in (0,\upsilon^+)$ such that $t_0+H(t_0) = c$. Setting $E = E_{t_0}$ we have 
$$\lambda(E) = \lambda(K) - c\quad{\rm  and}\quad \int_E f d\lambda = 0.
$$ 
It remains to verify that $E\subseteq K\cap (\inf K,\sup K)$.

Let $a = \min(a_{t_0},b_{t_0})$. Then the set 
$E\cap  \{f\neq 0\}$ has the empty intersection with the set 
$({\rm inf} K, {\rm inf} K + a)$ and also with the interval $({\rm sup} K -a,  {\rm sup} K)$. Due to the assumption 
$\lambda(\{f=0\}) = 0$, we may replace the set $E$ with the set $E' =  E \cap [{\rm inf} K + a,{\rm sup} K -a]$. Then, we have that
$E\setminus E' \subseteq \{f=0\}$ is a null set, and $E'$ is a compact set separated from the points ${\rm inf} K$ and ${\rm sup} K$. This completes the proof.
\end{proof}

\subsection{Rational splitting of the set $K$}
In this subsection, we shall assume that the compact set $K\subset [0,1]$ is presented as a union of two disjoint compact sets $K_1$ and $K_2$. In this case, we show below that the choice of the subset $E$ defined in the preceding lemma can be done in such a way as to ensure that the numbers $\frac{\lambda(E\cap K_1)}{\lambda(E)}$ and $\frac{\lambda(E\cap K_2)}{\lambda(E)}$ are rational.

\begin{lem}\label{Lemma:splitting1}
	Let $K_1$ and $K_2$ be compact subsets of $[0,1]$ with $\lambda(K_1\cap K_2)=0$ and assume that $K = K_1\cup K_2$ has positive measure. Let $f\in L_\infty(K)$ be a continuous, real-valued function satisfying assumptions (i) and (ii) of Theorem \ref{TheoremNowhereConstant}.   
Then, for $\epsilon>0$ we can find a compact subset $E\subseteq K$ of positive measure such that 
$\lambda(E) \geq \lambda(K) - \epsilon$ and 
$\int_{E}fd\lambda = 0$ and such that 
$$\frac{\lambda(E\cap K_1)}{\lambda(E)} = \frac{p}{q}$$ 
for some integers $p\ge 0$ and $q\ge 1$.
	\begin{proof}
		Let $K_1,K_2,f$ and $\epsilon$ satisfy the assumptions of Lemma \ref{Lemma:splitting1}. We can assume that both $\lambda(K_1),\lambda(K_2)>0$ otherwise we can take $E= K$. Further, if $f = 0$ we simply set $E = (K_1 \cap [0,r]) \cup K_2$ where $0\leq r\leq 1$ is chosen to satisfy $\lambda(E)\geq \lambda(K)-\epsilon$ and 
		$$\frac{\lambda(K_1\cap [0,r])}{\lambda(K_1\cap [0,r]) + \lambda(K_2)} = \frac{\lambda(E\cap K_1)}{\lambda(E)}\in\QQ.$$
Without loss of generality, we can assume that $f\not=0$. Writing 
$$
f_1 = f|_{K_1}\quad{\rm  and}\quad f_2 = f|_{K_2},
$$
we observe that $f_1,f_2\not=0$. Hence, letting 
$$\upsilon_1^\pm = \lambda(\{f_1^\pm> 0\})\quad{\rm
		and}\quad \upsilon_2^\pm = \lambda(\{f_2^\pm > 0\}),
$$
we have that $\upsilon_1^+>0$ or else $\upsilon_1^->0$ and likewise $\upsilon_2^+>0$, or else $\upsilon_2^->0$, due to the fact that $f_1,f_2\not=0$. 
Now if $\upsilon_1^+>0$ and $\upsilon_2^+>0$ then we must also have $\upsilon_1^->0$ or $\upsilon_2^->0$, again due to the fact that $f$ is mean zero. 
Hence we can assume that $\upsilon_1^+,\upsilon_2^->0$, the case that $\upsilon_1^-,\upsilon_2^+>0$ being similar by changing the roles of $K_1$ and $K_2$.
If $\lambda(\{f_1\geq \kappa\}) < \upsilon_1^+$, then 
we choose		$\omega$ to be the positive number $ \lambda(\{0< f_1< \kappa\})$.	In this case, we consider the function 
$$L:[0,\kappa]\to \RR\quad{\rm given\ by}\quad L(a) = \lambda(f_1^{-1}(0,a)).$$
 We have that $L$ is non-decreasing with $L(0) = 0$ and $L(\kappa) = \omega$. Now since $\lambda(f_1^{-1}(\{y\})) = 0$ for $y< \kappa$ we have that $L$ is continuous. Hence, for $r\in [0,\omega]$ we can find $0\leq a_r\leq \kappa$ such that for the set $A_r := f_1^{-1}(0,a_r)$ we have $\lambda(A_r) = L(a_r) =r$. Moreover, we can choose the largest possible $a_r$ with this property, that is
$$
a_r=\sup\{a:\ L(a)=r\}.
$$
Clearly, $L(a_r) = r$ as $L$ is continuous nondecreasing function.

Now, if $\lambda(\{f_1\geq \kappa\}) = \upsilon_1^+$, then we set 
$$\omega = \upsilon_1^+$$ 
and define 
the set 
$$A_r := \{f_1>0\} \cap [0,a_r), \quad {\rm for}\quad 0\leq r<\omega,
$$
where we choose $0\leq a_r\leq 1$ such that $\lambda(A_r) = r$. In both cases, we have that $A_r$ is open in $K_1$.
		
Further, we define the set 
$$
B_r = f_2^{-1}((-\infty,b_r))\subseteq K_2, \quad {\rm for}\quad0\leq r\leq \upsilon_2^-,
$$
where $b_r\leq 0$ is chosen  to satisfy $\lambda(B_r) = r$  and 
$$
b_r=\sup\{b:\ L(b)=r\}.
$$ 
Note that such a choice of  $b_r$ is possible thanks to the assumption that $\lambda(f^{-1}(\{y\})) = 0$ for all $y\leq 0$. 
We now have 
$$A_{r_1}\subset A_{r_2}\quad{\rm and}\quad B_{r_1}\subset B_{r_2}\quad{\rm whenever}\quad r_1<r_2,
$$ 
which imply that
\begin{equation}\label{a_r}
a_{r_1}< a_{r_2}\quad{\rm and}\quad b_{r_1}< b_{r_2}\quad{\rm whenever}\quad r_1<r_2,
\end{equation}

Let 
$$F_1,F_2:[0,\min\{\omega,\upsilon_2^-\}]\to \RR^+\quad{\rm be\ given\ by}\quad F_1(r) = \int_{A_r}fd\lambda\quad{\rm
		and}\quad F_2(r) = -\int_{B_r}fd\lambda.
$$ 
These are continuous, strictly increasing functions with $F_1(0) = F_2(0) = 0$. Let us compute their derivatives. In the case when $A_r = \{f_1>0\}\cap [0,a_r)$, we have for $h>0$ that
		$$\frac{F_1(r+h) - F_1(r)}{h}
		=\frac{\int_{A_{r+h}\setminus A_r}fd\lambda}{h} = \frac{\kappa h}{h} = \kappa.$$
		Hence, $F_1'(r) = \kappa$. On the other hand, in the case when $A_r = f_1^{-1}(0,a_r)$, we have 
		\begin{align*}
		\left| \frac{F_1(r+h) - F_1(r)}{h} -a_r\right|
		&=\left|\frac{\int_{A_{r+h}\setminus A_r}f - a_r d\lambda}{h}\right|\\
		&\leq\frac{\int_{A_{r+h}\setminus A_r}|f -a_r|d\lambda}{h} \leq |a_{r+h} - a_r|
		\end{align*}
		Now, since $a_r$ is chosen maximally such that $\lambda(A_r) = r$, it follows that for any $\epsilon>0$ we have 
$$\lambda(f_1^{-1}(0,a_{r}+\epsilon))>\lambda(A_r) =r.$$
		Hence,  if $r<r+h<\lambda(f_1^{-1}(0,a_{r}+\epsilon))$ we have $a_r<a_{r+h}<a_r + \epsilon$, so that $|a_{r+h}-a_r|\to 0$ as $h\downarrow 0$. Hence, the right hand derivative  $F_1'(r^+) = a_r$ is strictly increasing.
		Now, we consider the right hand derivative of $F_2$. For $h>0$ we have
		\begin{align*}
		\left| \frac{F_2(r+h) - F_2(r)}{h} + b_r\right|
		&=\left|\frac{\int_{B_{r+h}\setminus B_r}b_r -f d\lambda}{h}\right|\\
		&\leq\frac{\int_{B_{r+h}\setminus B_r}|b_r -f|d\lambda}{h} \leq |b_{r+h} - b_r|
		\end{align*}
		Now, since $b_r$ is chosen maximally such that $\lambda(B_r) = r$,  it follows that 
$$\lambda(f_2^{-1}(-\infty,b_{r}+\epsilon))>\lambda(B_r) =r,\quad \forall \epsilon>0.
$$
		Hence if $r<r+h<\lambda(f_2^{-1}(-\infty,b_{r}+\epsilon))$ we have $b_r<b_{r+h}<b_r + \epsilon$, so that $|b_{r+h}-b_r|\to 0$ as $h\downarrow 0$. This means that for the right hand derivative we have  $F_2'(r^+) = -b_r$, which is strictly decreasing due to \eqref{a_r}.
		
		Now fix $0<r_0 < \min\{\omega,\upsilon_2^-\}$ and choose $0<t_0<\min\{\omega,\upsilon_2^-\}$ with $0< F_1(t_0)< F_2(r_0)$, which can be done since $F_1$ is continuous and $F_1,F_2$  are strictly increasing (indeed, since $\lambda(A_r)$ is strictly increasing and since$ f|_{A_r} > 0$ we have that $ F_1(r) =  \int_{A_r} f d\lambda$ is also strictly increasing; the argument for $F_2$ is the same).
		
Let 
$$G: [0,t_0]\times [0,r_0] \to \RR\quad{\rm be\ given\ by}\quad G(t,r) = F_2(r) - F_1(t)
$$ 
so that we have 
$$G(t,0) = F_2(0) - F_1(t)\leq 0\quad{\rm and}\quad
G(t,r_0) = F_2(r_0) - F_1(t) > 0\quad {\rm for}\quad t\in [0,t_0].
$$ 
Observe that $G(t,\cdot)$ is continuous, and therefore there exists  $0\leq x<r_0$ such that $G(t,x) = 0$. 
Further, since $G(t,\cdot)$ is strictly decreasing this $x$ is unique. Now define $H:[0,t_0]\to [0,r_0)$ as the unique value $H(t)\in [0,r_0)$ satisfying $G(t,H(t)) = 0$.
		
For $t\in [0,t_0]$ we now have 
\begin{equation}\label{F_2(H(t))}
F_2(H(t)) - F_1(t) = 0,
\end{equation}
or equivalently,
		$$\int_{A_t}fd\lambda + \int_{B_{H(t)}}fd\lambda = 0.$$
Let us now set 
$$E_t = K\setminus (A_t\cup B_{H(t)}).
$$ We observe that $E_t$ is a compact set since $A_t,B_{H(t)}$ are open in $K_1,K_2$ respectively, and that
combining the assumption and the preceding display we have
$$
\int_{E_t} f d\lambda = \int_K f d\lambda - 0 = 0.
$$
Now $\lambda(E_t) = \lambda(K) - (H(t) + t)$ (indeed, by the construction $\lambda(A_t) = t, \lambda(B_{H(t)}) = H(t)$, and the sets $A_t$ and $B_{H(t)}$ are disjoint and sit inside of the set $K$). 

Now, since the function $G$ is continuous, 
and strictly increasing in the first variable $r$ and strictly decreasing in the second variable $t$, 
it follows that the function $H$ is continuous and strictly increasing. 
Hence, recalling that $H(0) = 0$, we can select $t_1>0$ so small that 
$$\lambda(E_{t_1})\geq \lambda(K) - \epsilon\quad{\rm that\ is}\quad H(t_1) + t_1\leq \epsilon.
$$ 
Now, we let 
\begin{equation}\label{R(t)}
R(t) = \frac{\lambda(E_t\cap K_1)}{\lambda(E_t)} = \frac{\lambda(K_1) - t}{\lambda(K) - (t+H(t))},\quad {\rm for}\quad t\in [0,t_1].
\end{equation}
 
We have 
$$R(0) = \frac{\lambda(K_1)}{\lambda(K)}.$$ 
Now, suppose that $R$ is constant on $[0,t_1]$. In this case, substitution to the left hand side of \eqref{R(t)} the value of $R(0)$ and  solving for $H(t)$,  we obtain 
$$
H(t) = \frac{\lambda(K) - \lambda(K_1)}{\lambda(K_1)}t,\quad 
H'(t)=\frac{\lambda(K) - \lambda(K_1)}{\lambda(K_1)},
\quad t\in [0,t_1].
$$ 
Further, recalling \eqref{F_2(H(t))}, we have  $F_2(H(t)) = F_1(t)$  so that differentiating from the right this equality yields 
$$F_2'(H(t))H'(t) = F_1'(t),\quad t\in [0,t_1].
$$ 
Now, since $F_2'\circ H$ is strictly decreasing and since $F_1'$ is strictly increasing (see \eqref{a_r}) and since $H'$ is a positive constant, this yields a contradiction with the equality
$$F'_2(H(t))H'(t) = F'_1(t).
$$ 
This contradiction shows that $R$ is not constant on $[0,t_1]$. Hence, since $R$ is continuous we can find a $t_2\in [0,t_1]$ such that $R(t_2) = \frac{p}{q}$ for some  integer $p\ge 0$ and some positive integer $q$. 
Now setting $E = E_{t_2}$, we obtain 
$$\lambda(E)\geq \lambda(K) - \epsilon, \quad \int_{E}fd\lambda = 0,\quad{\rm and}\quad \frac{\lambda(E\cap K_1)}{\lambda(E)} = R(t_2) = \frac{p}{q}.$$ This completes the proof.	
	\end{proof}
\end{lem}

\begin{lem}\label{Lemma:splitting2}
	Let $K \subseteq [0,1]$ be a compact set of positive measure, and let
$$k^M = \frac{\inf K + \sup K}{2},\quad
K^L = K\cap [0,k^M],\quad K^R = K\cap [k^M,1].
$$
Let $f\in L_\infty(K)$ be a continuous real-valued function satisfying assumptions (i) and (ii) of Theorem \ref{TheoremNowhereConstant}.

Let $c$ be an arbitrary scalar from $(0,\lambda(K))$ provided that  $\lambda(K^L) = 0$ or if $\lambda(K^R) = 0$ or from $(0,\min\{\lambda(K^L),\lambda(K^R)\})$ provided
$\lambda(K^L) > 0$ and $\lambda(K^R) > 0$.
Then there exists a compact subset 
$$
E\subset K\cap (\inf K,\sup K)\quad{\rm with}\quad 
\frac{\lambda(E\cap K^L)}{\lambda(E)} = \frac{p}{q}$$
for some  integer $p\ge 0$ and some positive integer $q$ and such that 
$$\lambda(E) = \lambda(K) -c\quad{\rm and}\quad
\int_{E}fd\lambda = 0.$$
	\begin{proof}
		If $\lambda(K^L) = 0$, we apply Lemma \ref{Lemma:shrinking} to the set $K^R$ and obtain the compact subset 
$$E\subseteq K^R\cap (\inf K^R,\sup K^R)\subseteq K\cap (\inf K,\sup K)
$$
such that
$$ \int_{E}fd\lambda = 0\quad{\rm and}\quad\lambda(E) = \lambda(K^R) - c = \lambda(K) -c.
$$ 
Moreover $\frac{\lambda(E\cap K^L)}{\lambda(K)} = 0$ so that $E$ satisfies the assertion of the lemma. The similar argument holds when $\lambda(K^R)=0$ via interchanging the roles of $K^L$ and $K^R$.

We can thus assume that $\lambda(K^L),\lambda(K^R)>0$.
Now we apply Lemma \ref{Lemma:splitting1} when 
$$K=K^L\cup K^R$$ 
with $f$ and $\frac{c}{2}$ to obtain a compact subset 
$$\widetilde{K}\subseteq K\quad{\rm with} \quad \int_{\tilde{K}}fd\lambda = 0\quad{\rm and}\quad
\frac{\lambda(\tilde{K}\cap K^L)}{\lambda(\tilde{K})} = \frac{p}{q}$$ and 
$$\lambda(\tilde{K})\geq \lambda(K) - \frac{c}{2}.$$ 
We set 
$$\tilde{K}^L = \tilde{K} \cap K^L\quad{\rm and}\quad\tilde{K}^R = \tilde{K} \cap K^R.
$$ 
We also set 
$$\Delta = \lambda(K) - \lambda(\tilde{K})$$ 
so that we have $0<\Delta<c$. Furthermore we have 
$$\lambda(\widetilde{K}^L)=\lambda(K^L) - \lambda(K^L \cap (K\setminus\widetilde{K}))\geq \lambda(K^L) - \Delta$$ and likewise 
$$\lambda(\widetilde{K}^R)\geq \lambda(K^R) - \Delta.$$
By our choice, we have
$$0<(c-\Delta)<\min\{\lambda(K^L) - \Delta,\lambda(K^R) - \Delta\} \leq\min\{\lambda(\widetilde{K}^L),\lambda(\widetilde{K}^R)\}.
$$ 
Now set 

\begin{equation}\label{h^L}
h^L = f|_{\widetilde{K}^L} - \frac{1}{\lambda(\widetilde{K}^L)}\int_{\widetilde{K}^L}fd\lambda\quad{\rm and}\quad h^R = f|_{\widetilde{K}^R} - \frac{1}{\lambda(\widetilde{K}^R)}\int_{\widetilde{K}^R}fd\lambda.
\end{equation}

Since the function $f$ satisfies assumption (ii) of Theorem \ref{TheoremNowhereConstant}, the same holds for $h^L$ and $h^R$.
This observation shows  that we can apply Lemma \ref{Lemma:shrinking} to the set $\tilde{K}^L$, the function $h^L$ and the scalar $\frac{p}{q}(c-\Delta)$ and also to the set $\tilde{K}^R$, the function $h^R$ and the scalar $(1-\frac{p}{q})(c-\Delta)$. Such applications yield
compact subsets 
$$E^L\subseteq \tilde{K}^L \cap (\inf \widetilde{K}^L, \sup \widetilde{K}^L)\quad{\rm and}\quad E^R\subseteq \tilde{K}^R\cap (\inf \widetilde{K}^R, \sup \widetilde{K}^R)$$ 
with 
\begin{equation}\label{lambda(E^L)}
\lambda(E^L) = \lambda(\widetilde{K}^L) - (c-\Delta)\frac{p}{q}\quad{\rm and}\quad \lambda(E^R) = \lambda(\widetilde{K}^R) - (c-\Delta)(1-\frac{p}{q})
\end{equation}
and furthermore 
$$\int_{E^L}h^Ld\lambda = \int_{E^R}h^Rd\lambda = 0.$$
Substituting into equalities above definitions of $h^L$ and $h^R$ from \eqref{h^L}, we arrive at
$$\int_{E^L}fd\lambda = \frac{\lambda(E^L)}{\lambda(\widetilde{K}^L)}\int_{\widetilde{K}^L}fd\lambda\quad{\rm and}\quad \int_{E^R}fd\lambda = \frac{\lambda(E^R)}{\lambda(\tilde{K}^R)}\int_{\tilde{K}^R}fd\lambda.
$$
Now, we define a compact set $E$ by setting 
$$E= E^L\cup E^R\subseteq K\cap (\inf K,\sup K).
$$ We have that
		\begin{align*}
		\int_{E}fd\lambda &= \frac{\lambda(E^L)}{\lambda(\tilde{K}^L)}\int_{\tilde{K}^L}fd\lambda
		+ \frac{\lambda(E^R)}{\lambda(\tilde{K}^R)}\int_{\tilde{K}^R}fd\lambda\\
		&=(1-(c-\Delta)\frac{\frac{p}{q}}{\lambda(\tilde{K}^L)})\int_{\tilde{K}^L}fd\lambda
		+ (1-(c-\Delta)\frac{1 -\frac{p}{q}}{\lambda(\tilde{K}^R)})\int_{\tilde{K}^R}fd\lambda\\
		&=(1-\frac{(c-\Delta)}{\lambda(\tilde{K})})\int_{\tilde{K}^L}fd\lambda
		+ (1-\frac{(c-\Delta)}{\lambda(\tilde{K})})\int_{\tilde{K}^R}fd\lambda\\
		&=(1-\frac{(c-\Delta)}{\lambda(\tilde{K})})\int_{\tilde{K}^L\cup \tilde{K}^R}fd\lambda = 0
		\end{align*}
Furthermore 
$$\lambda(E) = \lambda(\tilde{K}^L) - (c-\Delta)\frac{p}{q} + \lambda(\tilde{K}^R) - (c-\Delta)(1-\frac{p}{q}) = \lambda(\tilde{K}) - (c-\Delta) = \lambda(K) - c.
$$
Finally, we claim that
$$\frac{\lambda(E^L)}{\lambda(\tilde{K}^L)} = \frac{\lambda(E^R)}{\lambda(\tilde{K}^R)}.$$

Indeed, by \eqref{lambda(E^L)}, we have
$$\frac{\lambda(E^L)}{\lambda(\tilde{K}^L)}=\frac{\lambda(\tilde{K}^L)-(c-\Delta)\frac{p}{q}}{\lambda(\tilde{K}^L)}=
    \frac{\lambda(\tilde{K})-(c-\Delta)}{\lambda(\tilde{K})},
    $$ 
due to the equality $\lambda(\tilde{K}^L)=\lambda(\tilde{K})\frac{p}{q}$. Similarly,
$$\frac{\lambda(E^R)}{\lambda(\tilde{K}^R)}=\frac{\lambda(\tilde{K})-(c-\Delta)}{\lambda(\tilde{K})}.
$$
Hence, recalling that $E\cap K^L=E^L$, we have
		\begin{align*}
		\frac{\lambda(E^L)}{\lambda(E)} &=\frac{\lambda(E^L)}{\lambda(E^L)+\lambda(E^R)} = \frac{\lambda(\tilde{K}^L)}{\lambda(\tilde{K}^L) +\lambda(\tilde{K}^R)} = \frac{\lambda(\tilde{K}^L)}{\lambda(\tilde{K})} = \frac{p}{q}.
		\end{align*}

	\end{proof}
\end{lem}

\subsection{Constructing towers of the sets $K_a$}
In this subsection, for a given set $K$ satisfying assumptions of Lemma \ref{Lemma:splitting2} and for given  $\epsilon\in (0,\lambda(K))$, we shall build a measurable set $C$ such that $\lambda(C)\geq \lambda(K) - \epsilon$ which is later used to construct a function $g\in L_\infty(C)$ and a measure preserving transformation mod0 $T$ of $C$ such that $f|_{C} = g\circ T - g$.

Below, we shall use the following notation. Fix a sequence of natural numbers $(m_n)_{n=0}^\infty$ such that $m_n\geq 2$.  For every $n\ge 1$ denote $\mathcal{E}_n = \prod_{j=1}^{n}\{1,...,m_{j-1}\}$. With every element $a\in \mathcal{E}_n$, we shall link the measurable set $K_a$ (a subset of a fixed measurable set $K$) and consider  the collection of sets $\{K_a\}_{a\in \mathcal{E}_n}$ for $n\in \NN$. Further, we denote 
$$C_n = \bigcup_{a\in \mathcal{E}_n}K_a,\quad{\rm and}\quad C = \bigcap_{n=1}^\infty C_n$$ and define the functions
$$f_n = \sum_{a\in \mathcal{E}_n}\frac{1}{\lambda(K_a\cap C)}\int_{K_a \cap C}fd\lambda\cdot \chi_{K_a\cap C}\in L_\infty(C).$$

\begin{lem}\label{Lemma:contruction-of-chains}
 
Suppose that the set $K\subseteq [0,1]$ and the function  
$f$ 
 satisfy the assumptions of Lemma \ref{Lemma:splitting2}. Then for every $\epsilon\in (0,\lambda(K))$ there exists a sequence of natural numbers $(m_n)_{n=0}^\infty$ as above such that the following properties hold for every $n\ge 1$:
	\begin{enumerate}
		\item For $a\in \mathcal{E}_n$ the set $K_a$ is a compact subset of $[0,1]$. For $a_1,a_2\in \mathcal{E}_n$, $a_1\neq a_2$ we have either ${\rm sup} K_{a_1}\leq {\rm inf} K_{a_2}$ or else ${\rm sup} K_{a_2}\leq {\rm inf} K_{a_1}$.
		\item If $a\in \mathcal{E}_{n}$ and $b\in \mathcal{E}_{n+1}$ are such that $a_j = b_j$ for $1\leq j\leq n$, then $K_b\subseteq K_a$.
		\item For $a,b\in \mathcal{E}_n$ we have that  $K_a\cap C$ and $K_b\cap C$ are disjoint whenever $a\not=b$.
		\item For $a,b\in \mathcal{E}_n$ the sets $K_a$ and $K_b$ have the same positive measure $M_n :=\lambda(K_a) = \lambda(K_b)>0$. Moreover, $\lambda(K_a\cap C) = \frac{\lambda(C)}{|\mathcal{E}_n|}$. 
		\item We have $\int_{C_n}f d\lambda = 0$, and furthermore $\int_C fd\lambda =0$
		\item We have $\lambda(C_n) \geq \lambda(K) -  (1-2^{-n})\epsilon>\lambda(K) - \epsilon$ and furthermore $\lambda(C)\geq \lambda(K) - \epsilon$
		\item For every chain $K_{c_1}\supset K_{c_2}\supset ....$ with $c_j \in  \mathcal{E}_j$ we have $\Diam(K_{c_j})\to 0$ as $j\to \infty$.
		\item The set $\{K_a\cap C: a\in \mathcal{E}_n \text{ for some } n\geq 1\}$ generates the Borel $\sigma$-algebra $\mathcal{B}(C)$.
		\item We have that $||f_j-f|_{C}||_\infty \to 0$ as $j\to \infty$.
	\end{enumerate} 
	\begin{proof}
		We will construct the sequence  $(m_n)_{n=0}^\infty$ and the sets $\{K_a\}_{a\in \mathcal{E}_n}$ inductively. For convenience we first set $\mathcal{E}_0 = \{\varepsilon\}$ where $\varepsilon$ denotes the empty tuple, and we define $K_{\varepsilon} = K$. We see that properties (1)-(6) hold for $n=0$. Now fix $n\geq 0$ and suppose that the sets $K_a$ with $a\in \mathcal{E}_n$ have been defined so that properties (1)-(6) hold for this $n$.

Fix $a\in \mathcal{E}_n$ and set 
$$K_a^L := K_a \cap [\inf K_a,\frac{\inf K_a + \sup K_a}{2}]\quad{\rm and}\quad K_a^R :=  K_a \cap [\frac{\inf K_a + \sup K_a}{2},\sup K_a].$$
Observe that ${\rm diam }(K_a^L)\leq \frac12 {\rm diam }(K_a)$ and ${\rm diam }(K_a^R)\leq \frac12 {\rm diam }(K_a)$.
 Choose $\epsilon_n>0$ with
		\begin{align}
		\epsilon_n&< \min\{\frac{\epsilon}{2^n |\mathcal{E}_n|}, M_n\}\\
		\epsilon_n &< \min\{\lambda(K_c^L),\lambda(K_c^R)\}
		\text{ for } c\in \mathcal{E}_n \text{ for which } \lambda(K_c^L),\lambda(K_c^R)>0
		\end{align} 
		Further set 
$$h_a =  f - \frac{1}{\lambda(K_a)}\int_{K_a}fd\lambda\quad{\rm and}\quad \kappa_a = \kappa - \frac{1}{\lambda(K_a)}\int_{K_a}fd\lambda.
$$
Then $h_a$ is a continuous function on $K$ with $\int_{K_a}h_a = 0$ and such that $\lambda(h_a^{-1}(\{y\})) = 0$ for all $y<\kappa_a$ and so that $h_a|_{\{h_a\geq \kappa_a\}} = \kappa_a$ is constant.
		Because of this, and by the choice of $\epsilon_n$, we can now apply Lemma \ref{Lemma:splitting2}  to the set $K_a = K_a^L\cup K_a^R$, the function $h_a$ and the scalar $\epsilon_n$ to obtain a compact subset $\widetilde{K}_a\subseteq K_a$ of measure 
$\lambda(\widetilde{K}_a)= \lambda(K_a) - \epsilon_n$ so that $\int_{\widetilde{K}_a}h_a = 0$ and so that if we set
		$\widetilde{K}_a^L = \widetilde{K}_a \cap K_a^L$ and 
		$\widetilde{K}_a^R = \widetilde{K}_a \cap K_a^R$ we have
		$\frac{\lambda(\widetilde{K}_a^L)}{\lambda(\widetilde{K}_a)} = \frac{p_a}{q_a}$ for some integer $p_a\ge 0$ and positive integer $q_a$.

		Now, set 
$$m_{n} = 2\prod_{a\in \mathcal{E}_n}q_a,\quad k_a=\frac{m_np_a}{q_a}.$$
We now select points 
$$x_a^0<x_a^1<\dots<x_a^{k_a}=\frac{\inf K_a + \sup K_a}{2}<\dots<x_a^{m_n}
$$
in $K_a$ so that for $1\leq i\leq m_n$ the sets 
$$K_a^i:=\widetilde{K}_a\cap [x_a^{i-1},x_a^{i}]$$ all have equal measure 
$$\lambda(K_a^i) = \frac{\lambda(\widetilde{K}_a)}{m_n} = \frac{\lambda(K_a) - \epsilon_n}{m_n} = \frac{M_n - \epsilon_n}{m_n}$$  and moreover  
$$K_a^i\subset \widetilde{K}_a^L,\quad \forall i\leq k_a$$ and 
$$
K_a^i\subset \widetilde{K}_a^R,\quad \forall k_a<i\leq m_n.
$$ 
		
		Now if $b = (a_1,a_2,...,a_n,i)$ with $1\leq i\leq m_{n}$ then we define $K_b=K_a^i$. 
		
Observe that the definitions of $K_a^i$ and that of $K_b$ guarantee that the assertions (1) and (2) of Lemma 	\ref{Lemma:contruction-of-chains} hold.

This completes the construction.\\

		Before we show that all stated properties hold, we first give an intuitive idea of what we have done.
		We had first shrunk the set $K_a$ to some compact subset $\widetilde{K}_a$. This has been done in such a way as to preserve the equality
\begin{equation}\label{equality of average}
\frac{1}{\lambda (K_a)}\int_{K_a}fd\lambda		=\frac{1}{\lambda (\widetilde{K}_a)}\int_{\widetilde{K}_a}fd\lambda
\end{equation}				
and keep the ratio 
$$
\frac{\lambda(K_a^L\cap \widetilde{K}_a)}{\lambda(\widetilde{K}_a)}
$$
rational. Thereafter we could choose $m_n$, and divide the set $K_a$ from left to right in sets $K_a^i$. 
Due the choice of $m_n$,  the subsets $K_a^i$  are either contained in $K_a^L$ or in $K_a^R$. Now, for $b = (b_1,..,b_n,i)$ we have defined $K_b$ equal to some $K_a^i$. 
This guarantees that 
${\rm diam}(K_b)\leq \frac12{\rm diam}(K_a)$ and thus the diameter of the tower goes to zero, which is the assertion (7).

It is also important to emphasize that Lemma \ref{Lemma:splitting2} guarantees that the sets $\widetilde{K}_a^L$ and $\widetilde{K}_a^R$ do not contain ${\rm inf}K_a$ and ${\rm sup}K_a$.		Therefore, 
$$
{\rm inf}K_a,\ {\rm sup} K_a\notin 
\cup_{b\in \mathcal{E}_{n+1}} K_b,
$$
and therefore, the family $\{K_a\cap C_{n+1}, a\in \mathcal{E}_n\}$ consists of pairwise disjoint sets. This observation guarantees that the assertion (3) of Lemma \ref{Lemma:contruction-of-chains} holds.


		\begin{figure}[!h]\label{Fig:visualisation-subset-structure}
			\title{{Visualisation of subset structure}}\par\medskip
			\begin{tikzpicture}	
			\node (A) at (0,0) {$K_a = K_a^L\cup K_a^R$};
			\node (B) at (0,-1) {$\widetilde{K}_a = \widetilde{K}_a^L\cup \widetilde{K}_a^R$};		
			
			\node (C1) at (-3,-3) {$K_{(a_1,..,a_n)}^1$};
			\node (C2) at (-1,-3) {...};
			\node (C3) at (1,-3) {...};
			\node (C4) at (3,-3) {$K_{(a_1,..,a_n)}^{m_n}$};
			
			\node (D1) at (-3,-4) {$K_{(a_1,..,a_n,1)}$};
			\node (D4) at (3,-4) {$K_{(a_1,..,a_n,m_n)}$};


			\path [-] (A) edge node[left,blue,thick] {} (B);	
			
			\path [-] (B) edge node[left,blue,thick] {} (C1);	
			\path [-] (B) edge node[left,blue,thick] {} (C2);	
			\path [-] (B) edge node[left,blue,thick] {} (C3);	
			\path [-] (B) edge node[left,blue,thick] {} (C4);	
			
			\path [-] (C1) edge node[left,blue,thick] {} (D1);	
			\path [-] (C4) edge node[left,blue,thick] {} (D4);					
			\end{tikzpicture}
			
			\caption{Visualisation of the subdivision of $K_a$, for some $a=(a_1,..,a_n)\in \mathcal{E}_n$, in subsets. The lines mean that the lower set is included in the upper set. A subset $K_b\subseteq K_a$ with $b\in \mathcal{E}_{n+1}$ is set equal, to the set $K_a^i$ for some $1\leq i\leq m_n$. Further, $K_b$ is either fully contained in $K_a^L$ or in $K_a^R$, hence its diameter is less than half the diameter of $K_a$. This ensures that the diameter of elements of $n$-level in the tower tends to zero as $n\to \infty$.}
		\end{figure}
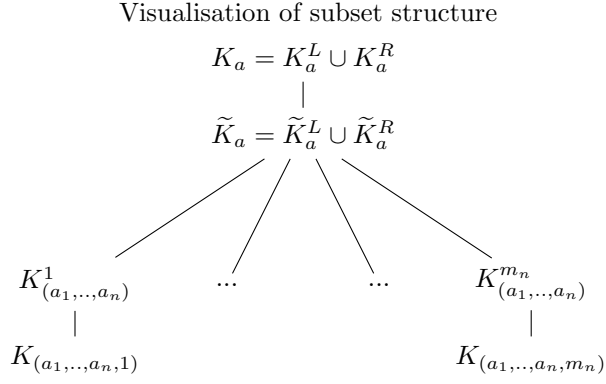	
\medskip 

Now we shall verify the remaining assertions of Lemma  \ref{Lemma:contruction-of-chains} (recall that we have verified (1)- (3) and (7) above).
	


		
		(4) Fix $n\geq 0$ and observe that all sets $\{K_g\}_{g\in \mathcal{E}_n}$ have equal positive measure $M_n$. Choose $a,b\in \mathcal{E}_n$ and $c,d\in \mathcal{E}_{n+1}$ such that $K_c\subseteq K_a$ and $K_d\subseteq K_b$.
		We then have 
$$\lambda(K_c)  = \frac{M_n - \epsilon_n}{m_n}  = \lambda(K_d).$$
		Further, since $\epsilon_n<M_n$ this measure is positive.
		Hence,  by induction all sets $\{K_g\}_{g\in \mathcal{E}_{n+1}}$ have equal positive measure. 
Finally, for any $a\in \mathcal{E}_n$, we have 
$$\lambda(K_a\cap C) = \lim\limits_{N\to\infty}\lambda(K_a\cap C_N) = \lim\limits_{N\to\infty}\frac{|\mathcal{E}_N|}{|\mathcal{E}_n|}M_N =
		\lim\limits_{N\to\infty}\frac{\lambda(C_N)}{|\mathcal{E}_n|} = 
		\frac{\lambda(C)}{|\mathcal{E}_n|}.$$\\
		
		(5) For $n=0$, we have $\int_{C_0}fd\lambda = \int_{K}fd\lambda = 0$. Now choosing integer $n\ge 0$ and assuming that $\int_{C_n}fd\lambda = 0$,  we have (the second equality below follows from equality \eqref{equality of average} above):
		
		\begin{align*}
		\int_{C_{n+1}}fd\lambda &=\sum_{a\in \mathcal{E}_{n}}\int_{\widetilde{K}_a}fd\lambda\\
		&=\sum_{a\in \mathcal{E}_n} \frac{\lambda(\widetilde{K}_a)}{\lambda(K_a)}\int_{K_a}fd\lambda d\lambda\\
		&=\sum_{a\in \mathcal{E}_n} \frac{\lambda(K_a) - \epsilon_n}{\lambda(K_a)}\int_{K_a}fd\lambda d\lambda\\
		&=\frac{M_n - \epsilon_n}{M_n}\sum_{a\in \mathcal{E}_n} \int_{K_a}fd\lambda d\lambda = \frac{M_n - \epsilon_n}{M_n}\int_{C_n}fd\lambda d\lambda = 0
		\end{align*} 
		Hence, inductively we have $\int_{C_n}fd\lambda = 0$ for all $n\in \NN$.
		Now moreover we have that $\left|\int_{C}fd\lambda\right| \leq  \left|\int_{C_n}fd\lambda\right| + \lambda(C_n\setminus C)||f||_\infty = \lambda(C_n\setminus C)||f||_\infty \to 0$ as $n\to \infty$.\\
		
		(6) We have $\lambda(C_0) = \lambda(K)$. Now, choose $n\geq 0$ and assume $\lambda(C_n) \geq \lambda(K) - (1-2^{-n})\epsilon$. Then for $a\in \mathcal{E}_n$ we have $\lambda(K_a)\geq \frac{\lambda(K) - (1 - 2^{-n})\epsilon}{|\mathcal{E}_n|}$. Now for $K_b\subseteq K_a$ with $b\in \mathcal{E}_{n+1}$ we have (due to the assumption on $\epsilon_n$)
\begin{align*}
M_{n+1} &= \lambda(K_b) \cr &= \frac{\lambda(\widetilde{K}_a)}{m_n} =
		\frac{\lambda(K_a) - \epsilon_n}{m_n} \geq \frac{\lambda(K_a)}{m_{n}} -\frac{\epsilon}{2^{n+1}|\mathcal{E}_n|m_{n}}\geq \frac{\lambda(K)}{m_n|\mathcal{E}_n|} - \frac{(1 - 2^{-n})\epsilon}{m_n|\mathcal{E}_n|} - \frac{2^{-(n+1)}\epsilon}{m_n|\mathcal{E}_n|} \cr &= \frac{\lambda(K) - (1-2^{-(n+1)})\epsilon}{|\mathcal{E}_{n+1}|}.
\end{align*}
		Hence 
$$\lambda(C_{n+1})= |\mathcal{E}_{n+1}|M_{n+1} \geq \lambda(K) - (1-2^{-(n+1)})\epsilon.$$ 
This proves the first claim by induction.
Furthermore, we have also obtained that
$$\lambda(C) = \inf_{n\in\NN} \lambda(C_n)\geq \lambda(K) - \epsilon.$$\
\
		
		(8)  We show that $\mathcal{A} = \{K_a\cap C: a\in \mathcal{E}_n,\ n\ge 0\}$ generates $\mathcal{B}(C)$. 
First of all, since for $n\ge 0$ the sets 
$\{K_a \cap C\}_{a\in \mathcal{E}_n}$ are compact, they are contained in the Borel $\sigma$-algebra $\mathcal{B}(C)$. Now choose $u\in [0,1]$, we show that $[0,u)\cap C$ is generated by $\mathcal{A}$. Let $x\in [0,u)\cap C$.
We have
\begin{equation}\label{diam}
{\rm diam}(K_a)\to 0,\quad a\in {\mathcal E}_n,\quad n\to \infty.
\end{equation}
Hence, there exist $N\in \NN$ and  $a_x\in \mathcal{E}_N$ such that $K_{a_x}$ contains $x$ and such that $\sup K_{a_x}<u$. Now let $\mathcal{A}_0 := \{K_{a_x}\cap C: x\in C\cap [0,u)\}\subseteq \mathcal{A}$ which is countable, due to the fact that every ${\mathcal E}_n$ is finite. Now $\bigcup_{A\in \mathcal{A}_0}A = C\cap [0,u)$ is generated by $\mathcal{A}$. Now, since $\{[0,u)\cap C: u\in [0,1]\}$ generates the Borel $\sigma$-algebra $\mathcal{B}(C)$, we also have that $\mathcal{A}$ generates $\mathcal{B}(C)$.\\
		
		(9) Since $f$ is continuous on $K$ and since $K$ is compact, we have that $f$ is uniformly continuous on $K$. Hence, for $\epsilon'>0$ we can find a $\delta>0$ such that $|f(x) - f(y)|<\epsilon'$ whenever $|x-y|<\delta$. Appealing to \eqref{diam}, 
		 we can find a $N\in \NN$ such for $n\geq N$ we have ${\rm diam}(K_a)<\delta$ for all $a\in \mathcal{E}_n$. Hence we have $|f(x) - f(y)|< \epsilon'$ for $x,y\in K_a$ and $a\in \mathcal{E}_n$. Now for $x\in K_a\cap C$ we have 
$$|f_n(x) - f(x)|= \left|\frac{1}{\lambda(K_a\cap C)}\int_{K_a\cap C}f(t) - f(x)dt\right| 	\leq \frac{1}{\lambda(K_a\cap C)}\int_{K_a\cap C}|f(t) - f(x)|dt
		\leq \epsilon'.
$$ Now, this holds for all $x\in C$ and $\epsilon'>0$, thus $\|f_n - f|_{C}\|_\infty \to 0$ as $n\to \infty$.
	\end{proof}
\end{lem}

Now, we shall pass to the construction of the function $g$ and  measure preserving transformation mod0 $T$ of the set $C$ constructed in Lemma \ref{Lemma:contruction-of-chains}. 

\begin{lem}\label{Lemma:solving-equation-on-subset}
Suppose that the set $K\subseteq [0,1]$ and the function  
$f$ satisfy the assumptions of Lemma \ref{Lemma:splitting2} (and Lemma \ref{Lemma:contruction-of-chains}). Take $\epsilon\in (0,\lambda(K))$ and take the set $C$ from Lemma \ref{Lemma:contruction-of-chains}. Then we can find a function $g\in L_\infty(C)$ with $||g||_\infty \leq (1+\epsilon)||f||_\infty$ and a measure preserving transformation mod0 $T$ of $C$ such that $f|_{C} = g\circ T - g$.
\end{lem} 
\begin{proof} 
 
We shall use notation introduced at the beginning of this subsection and in the formulation of Lemma \ref{Lemma:contruction-of-chains}. 
We now let 
$$v_n:\mathcal{E}_n\to \{1,...,|\mathcal{E}_n|\}$$ 
be the function that arranges the elements in 
$\mathcal{E}_n$ in lexicographical order.
	Further, for $i\in \{1,..,|\mathcal{E}_n|\}$ let 
	$$I_i^n = K_{v_n^{-1}(i)}\cap C$$
	which is compact (see assertion (1) in Lemma \ref{Lemma:contruction-of-chains}).
Since $\|f_n - f|_{C}\|_\infty \to 0$ as $n\to \infty$ (see assertion (9) in Lemma \ref{Lemma:contruction-of-chains}), 
it follows that there exists a sequence $(n_k)_{k\geq 0}$ of natural numbers such that
	for $n\geq n_k$ we have
$$\|f_n-f|_{C}\|_\infty 
\leq 2^{-k-3}\epsilon||f||_\infty.
$$
Setting, 
$$h_k = f_{n_k} - f_{n_{k-1}},\quad{\rm so\ that}\quad 
\|h_k\|_\infty\leq 2^{-k-2}\epsilon\|f\|_\infty,\quad k\ge 1$$ we  have 
$$f = f_{n_0} + \sum_{k=1}^\infty h_k.$$

Now, for $f_{n_0}$ let us denote by $a_i$ the value of $f_{n_0}$ taken on $I_i^{n_0}$ for 
$1\leq i\leq |\mathcal{E}_{n_0}|$. 
As $\int_{C}fd\lambda =0$ we have 
$\sum_{i=1}^{|\mathcal{E}_{n_0}|} a_i = 0$ so that we can use Lemma \ref{sum} to obtain a cyclic permutation 
$\sigma$ of $\{1,..,|\mathcal{E}_{n_0}|\}$ so that 
$\left|\sum_{i=1}^{m}a_{\sigma(i)}\right|\leq \max\{a_i: 1\leq i\leq |\mathcal{E}_{n_0}|\} \leq \|f_{n_0}\|_\infty$ for $0\leq m\leq |\mathcal{E}_{n_0}|$. 
Now, denote by $T_0$ the measure preserving transformation mod0 of $C$ sending $I_{\sigma(i)}^{n_0}$ to 
$I_{\sigma(i+1)}^{n_0}$ for 
$1\leq i\leq |\mathcal{E}_{n_0}| -1$ and
 sending $I_{\sigma(|\mathcal{E}_{n_0}|)}^{n_0}$ to $I_{\sigma(1)}^{n_0}$. 
Such measure preserving transformation mod0 exists by Theorem \ref{Prelim:isomorphism}, since all sets $I_{i}^{n_0}$ for $i=1,...,|\mathcal{E}_n|$ have equal measure. We now denote by $g_0:C\to \RR$ the function, taking on $I_{\sigma(l)}$ the value $\sum_{i=1}^{l-1}a_{\sigma(i)}$ for $l=2,...,|\mathcal{E}_{n_0}|$ and taking value $0$ on 
the interval $I_{\sigma(1)}$. 
Then $\|g_0\|_\infty\leq\|f_{n_0}\|_\infty\leq\|f\|_\infty$ and for $l=2,...,|\mathcal{E}_n|$ and $t\in I_{\sigma(l)}$ we have $g_0(T_0(t)) - g_0(t) = \sum_{i=1}^{l}a_{\sigma(i)} - \sum_{i=1}^{l-1}a_{\sigma(i)} = a_{\sigma(l)} = f_{n_0}(t)$. When $l=1$ and $t\in I_{\sigma(1)}$, we have  $g_0(T_0(t)) - g_0(t) = \sum _{i=1}^1 a_{\sigma(i)} -0=a_{\sigma(1)}=f_{n_0}(t)$.\\

%
	
Using the same argument as in \cite{Kwapien},  for each $k\ge 1$, we denote $J_k = \{I_i^{n_k}: 1\leq i\leq |\mathcal{E}_{n_k}|\}$, and define a sequence $\{T_k\}_{k=1}^\infty$ of measure preserving transformations mod0 $T_k$ of $C$  and functions $\{g_k\}_{k=1}^\infty $ with   $g_k\in L_\infty(C)$ satisfying the following:
	\begin{enumerate}[(i)]
		\item $T_k$ is a cyclic rearrangement of the sets of $J_k$.
		\item $T_{k+1}$ extends $T_k$ in the sense that if $I\in J_{k}$, $I'\in J_{k+1}$ and $I'\subseteq I$ then $T_{k+1}(I')\subseteq T_k(I)$
		\item $||g_k||_\infty\leq 4||h_k||_\infty$
		\item $g_k$ is constant on all the sets $I\in J_k$
		\item $h_k = g_k\circ T_k - g_k$ on $C$
	\end{enumerate}

Now, we suppose that the transformations $T_1,...,T_k$ and functions $g_1,...,g_k$ with given properties have been already defined. For convenience we set $n = |J_k|$ and $m = \frac{|J_{k+1}|}{|J_k|}$.
	Let $I_1,I_2,\dots,I_{n}$ be the sets from $J_k$, enumerated so that $T_k(I_i)=I_{i+1}$ when $i<n$ and $T_k(I_n)=I_1$, which can be done since $T_k$ is a cyclic rearrangement of the sets of $J_k$.
	Furthermore, for $i=1,2,\dots,n$ and $j=1,2,\dots,m$ let us denote by $I_{i,j}$ all sets from $J_{k+1}$ which are contained in $I_i$. Denote by $a_{i,j}$ the value of the function $h_{k+1}$ on $I_{i,j}$. Since 
$$\int_{I_i}h_{k+1}d\lambda=\sum_{j=1}^{m}\int_{I_{i,j}}f_{n_{k+1}} - f_{n_k}d\lambda= 0,\quad \forall I_i\in J_k,$$ 
 it follows that $\sum_{j=1}^m a_{i,j}=0$ for all $i=1,\dots,n$. In addition, $|a_{i,j}|\leq\|h_{k+1}\|_\infty$ for all $i=1,\dots,n,\ j=1,\dots,m$.
	Therefore, by Lemma \ref{Prelim:permutations-kwapien} it follows that there exist such rearrangements $\sigma_1,\dots,\sigma_n$ of the numbers $\{1,\dots,m\}$ that $$|\sum_{i=1}^k a_{i,\sigma_i(j)}|\leq 2\|h_{k+1}\|_\infty$$
	for $k=1,\dots,n$ and $j=1,\dots,m$. Define $T_{k+1}$, by setting 
$$T_{k+1}(I_{i,\sigma_i(j)})=I_{i+1,\sigma_{i+1}(j)},\quad i=1,\dots,n-1,\quad j=1,\dots,m .$$
We set 
$$b_j=\sum_{i=1}^n a_{i,\sigma_i(j)},\quad j=1,\dots,m.$$ Since $\sum_{j=1}^m b_j=\sum_{i=1}^n\sum_{j=1}^m a_{i,j}=0$ and $|b_j|\leq 2\|h_{k+1}\|_\infty$, Lemma \ref{sum} yields the existence of the rearrangement $\sigma_0$ of the numbers $1,\dots,m$ such that $$|\sum_{j=1}^l b_{\sigma_0(j)}|\leq 2\|h_{k+1}\|_\infty,\quad \forall l=1,\dots,m.$$
Set 
$$
T_{k+1}(I_{n,\sigma_n(\sigma_0(j))})=I_{1,\sigma_1(\sigma_0(j+1))},\quad \forall j=1,\dots,m-1$$ and set 
$$
T_{k+1}(I_{n,\sigma_n(\sigma_0(m))})=I_{1,\sigma_1(\sigma_0(1))}.$$ 
Observe that $T_{k+1}$ is a measure preserving transformation mod0 due to Theorem \ref{Prelim:isomorphism} (taking into account that the sets $I_{i,j}$ for $i=1,...,n$ and $j=1,...,m$ are of equal positive measure).
	
	Let us explain in a simpler language what we have just done. The matrix
	$(a_{i,j})_{n\times m}$ is transformed into the matrix $(a_{i,\sigma_i(\sigma_0(j))})_{n\times m}$ in such a way that in every column the module of the sum of all first elements $k$ does not exceed 
	$2\|h_{k+1}\|_\infty$ and the sum of the first $l$ columns does not exceed  $2\|h_{k+1}\|_\infty$. Next, we have build the transformation $T_{k+1}$, which "scans" the matrix column-wise: the first column from top to bottom, then the second column from top to bottom, etc.
	
Hence, 
$$|\sum_{r=0}^l h_{k+1}(T_{k+1}^r(t))|=|\sum_{j=1}^{p-1} b_{\sigma_0(j)}+\sum_{i=1}^q a_{i,\sigma_i(\sigma_0(p))}|\leq 4\|h_{k+1}\|_\infty,$$ 
where $l+1=(p-1)n+q$, for every $t\in I_{1,\sigma_1(\sigma_0(1))}$ and every $l=0,\dots,nm-1$.
	
	Now, let us define the function $g_{k+1}$ by setting its value on
	$T_{k+1}^l(I_{1,\sigma_1(\sigma_0(1))})$ equal to $\sum_{r=0}^{l-1}
	h_{k+1}(T_{k+1}^r(t))$, where $t\in
	I_{1,\sigma_1(\sigma_0(1))}$ for $l=1,...,nm-1$ and setting
	$g_{k+1}(I_{1,\sigma_1(\sigma_0(1))})=0$. Then we have
	$\|g_{k+1}\|_\infty\leq 4\|h_{k+1}\|_\infty$.
	
	Let $t\in I_{1,\sigma_1(\sigma_0(1))}$.
	If
	$0<l<nm-1$, then we have
$$g_{k+1}(T_{k+1}(T_{k+1}^l(t)))-g_{k+1}(T_{k+1}^l(t))=\sum_{r=0}^l
	h_{k+1}(T_{k+1}^r(t))-\sum_{r=0}^{l-1}
	h_{k+1}(T_{k+1}^r(t))=h_{k+1}(T_{k+1}^l(t)),$$ and further
$$g_{k+1}(T_{k+1}(t))-g_{k+1}(t)=h_{k+1}(t)-0=h_{k+1}(t)$$ finally yielding
$$g_{k+1}(T_{k+1}(T_{k+1}^{nm-1}(t)))-g_{k+1}(T_{k+1}^{nm-1}(t))=0 -
	\sum_{r=0}^{nm-2}h_{k+1}(T_{k+1}^r(t))=
	h_{k+1}(T_{k+1}^{nm-1}(t)).$$
Thus, for every $t\in C$ we have 
$$g_{k+1}(T_{k+1}(t))-g_{k+1}(t)=h_{k+1}(t).$$
This completes the construction of the functions 
$\{g_k\}_{k=1}^\infty$ and transformations 
$\{T_k\}_{k=1}^\infty$ with required properties.\\
	
	It follows from the construction that $T_{k+1}$ satisfies the condition $(ii)$. Hence the sequences $T_k$ and $g_k$ satisfy the conditions $(i)-(v)$. Observe that the inverse mappings $T_k^{-1}$ also satisfy the condition $(ii)$.
	
	It follows from the condition $(iii)$ that the series $\sum_{k=0}^\infty g_k$ converges in $L_\infty(C)$ to some function $g$ and $\|g\|_\infty\leq \|g_0\|_\infty + \sum_{k=1}^{\infty}\|g_k\|_\infty \leq \|f\|_\infty + \epsilon||f||_\infty = (1+\epsilon)\|f\|_\infty$.\\
	
	Next, it follows from $(ii)$ that for almost all $t\in C$ the sequence $T_k(t)$ is Cauchy and hence it converges. We then set $T(t)=\lim_{k\rightarrow\infty} T_k(t)\in C$.
	Now, if $x,y\in C$ with $x\not=y$ then there is a $N$ such that  $x\in I$ and $y\in \widetilde{I}$ for some $I,\widetilde{I}\in J_N$ with $I\not= \widetilde{I}$. Hence, $T(x)\in T(I)\subseteq T_k(I)$ and $T(y)\in T(\widetilde{I}) \subseteq T_k(\widetilde{I})$ so that $T(x)\not=T(y)$. Hence, $T$ is injective. 
	
Denote
$$
A_k=\bigcup_{I\in J_k} T_k^{-1}(I),\quad k\ge 1.
$$	
For each $k\ge 1$ the set $A_k$ has full measure in $C$. Hence,
$$
\bigcap_{k=1}^\infty A_k
$$	
also has full measure in $C$.

Now, let $\omega\in \bigcap_{k=1}^\infty A_k$, then, for $k\geq 1$ we can find $I_\omega^k\in J_k$ with $\omega\in I_\omega^k$. Therefore, for all such $\omega$, we have
	$\bigcap_{k=1}^\infty T_k^{-1}(I_{\omega}^{k})$ is non-empty. 
Hence, we can find $x\in\bigcap_{k=1}^\infty T_k^{-1}(I_{\omega}^{k})\subseteq C$ so that $T_k(x)\in I_\omega^k$ for all $k\geq 1$. 
Now, as $\Diam(I_\omega^k)\to 0$ (see assertion (7) in Lemma \ref{Lemma:contruction-of-chains}) we must have $T(x) = \omega$. This means that $T$ is a bijection between two subsets of the set $C$ of full measure: one is the set $ \bigcap_{k=1}^\infty A_k$ and the other is its image under the mapping $T$.
	
Let us verify that $T$ is measure preserving. Fix $k\geq 1$, then for $I\in J_k$ we have $T(I) = \widetilde{I}$ for some $\widetilde{I}\in J_k$. Hence, 
$$\lambda(T(I)) = \lambda(\widetilde{I}) = \frac{\lambda(C)}{|J_k|} = \lambda(I),\quad \forall I\in J_k.$$ Now as $\bigcup_{n=1}^\infty J_k$ generates the Borel $\sigma$-algebra on $C$ (see assertion (8) in Lemma \ref{Lemma:contruction-of-chains}), this equality holds for all sets in $\mathcal{B}(C)$. Thus $T$ is a measure preserving transformation mod0 of $C$.
	Now, we have for $k\geq 1$ that 
$$g_k(T(x)) - g_k(x) = g_k(T_k(x)) - g_k(x) = h_k(x).$$
	Hence, 
$$g(T(x)) - g(x) = \sum_{k=0}^\infty g_k(T(x)) - g_k(x) = f_{n_0} + \sum_{k=1}^{\infty}h_k = f.$$ 
This completes the proof of Lemma \ref{Lemma:solving-equation-on-subset}.
\end{proof}

\subsection{Completing the proof of Theorem \ref{TheoremNowhereConstant}}
We are now in a position to complete the proof of Theorem \ref{TheoremNowhereConstant}.
Let $K\subseteq [0,1]$ and $f$ satisfying the assumptions of Theorem \ref{TheoremNowhereConstant} be given. Further choose $\epsilon>0$. We will define a measure preserving transformation mod0 $T$ of $K$ and a function $g\in L_\infty(K)$ with $||g||_\infty\leq (1+\epsilon)||f||_\infty$ such that $f = g\circ T- g$.
We will do this by considering a family of measurable subsets $\{A_i\}_{j\in \mathcal{J}}$ of $K$, where $\mathcal{J}$ is some index set, so that the following holds:
\begin{enumerate}
	\item For different $i,j\in \mathcal{J}$ we have that $A_i$, $A_j$ are disjoint.
	\item For $j\in \mathcal{J}$ we have $\int_{A_j}fd\lambda =0$ and $\lambda(A_j)>0$.
	\item For $j\in \mathcal{J}$ we have that $f|_{A_j}$ is continuous.
	\item For $j\in \mathcal{J}$ we can find a function $g_j\in L_\infty(A_j)$ with $||g_j||_\infty \leq (1+\epsilon)||f||_\infty$ and a measure preserving transformation mod0 $T_j$ of $A_j$ such that $f|_{A_j} = g_j\circ T_j -g_j$.
\end{enumerate}
We can equip the set of of such selections $\{A_j\}_{j\in \mathcal{J}}$ with the partial ordering $\subseteq$ of being a subset. Now, suppose that we have a chain 
$\{\{A_i\}_{i\in \mathcal{J}_i}\}_{i\in I}$ for some index set $I$, then if we let $\mathcal{J} = \bigcup_{i\in I}\mathcal{J}_i$, we obtain an upper bound
$\{A_j\}_{j\in \mathcal{J}}$ for the chain. Hence, every chain has an upper bound, so that we can apply Zorn's Lemma. 
Now, we can choose a selection
$\{A_j\}_{j\in \mathcal{J}}$ that is maximal. Now, let $D =K\setminus\bigcup_{j\in \mathcal{J}}A_j$ and suppose
$\lambda(D)>0$. We have $\int_{D}fd\lambda = \int_{K}fd\lambda  - \sum_{j\in\mathcal{J}}\int_{A_j}fd\lambda = 0$. We set $\tau^\pm = \lambda(\{f|_{D}^\pm\geq \frac{1}{2}||f|_{D}^\pm||\})$ and set
$z = (1 + 2||f|_{D}||_\infty\max\{\frac{1}{||f|_{D}^+ ||_\infty},\frac{1}{||f|_{D}^-||_\infty}\})$.
We can now choose $\epsilon_1>0$  with $\epsilon_1<\min\{\frac{\lambda(D)}{z},\tau^+\frac{||f^+|_{D}||_\infty}{||f|_{D}||_\infty},\tau^-\frac{||f^-|_{D}||_\infty}{||f|_{D}||_\infty}\}$.
Now, we can apply Lusin's theorem, Theorem \ref{Prelim:lusin} on $D$, to get a compact set 
$E\subseteq D$ of measure $\lambda(E)>\lambda(D) - \epsilon_1$  so that $f$ is continuous on $E$. Now, by the bound on $\epsilon_1$ we can then apply Lemma \ref{Lemma:compensation-of-integral} on $E\subseteq D$ with $f|_{D}$ and $\epsilon_1$ to select a compact subset $K\subseteq E$ of measure $\lambda(K)>\lambda(D) - z\epsilon_1> 0$  such that $\int_{K}fd\lambda = 0$. 
Applying Lemma \ref{Lemma:contruction-of-chains}, together with Lemma \ref{Lemma:solving-equation-on-subset}, to $K$, $f$ and $\min\{\epsilon_1,\frac{1}{2}\lambda(K)\}$ we obtain a compact subset $C\subseteq K$ with 
$\lambda(C)\geq \lambda(K) - \frac{1}{2}\lambda(K)>0$ and such that $\int_{C}fd\lambda = 0$, a function $g\in L_\infty(C)$ with 
$\|g\|_\infty \leq (1+\epsilon)||f||_\infty$ and a measure preserving transformation mod0 $T$ of $C$ such that $f|_{C} = g\circ T -g$. We now see that $\{C\}\cup \{A_j\}_{j\in \mathcal{J}}$ satisfies properties (1)-(4) above so that $\{A_j\}_{j\in \mathcal{J}}$ is not maximal, which is a contradiction. We conclude that $\lambda(D) = 0$. \\

\sloppy Having established that $K\setminus\bigcup_{j\in \mathcal{J}}A_j$ has measure zero, we can define the final transformation $T$ of $K$ as $T|_{A_j} = T_j$ and $T(x) =x$ for $x\in K\setminus\bigcup_{j\in \mathcal{J}}^\infty A_j$, and likewise define the function $g$ as $g|_{A_j} = g_j$. Then $T$ is a measure preserving transformation mod0 of $K$ and $g$ is a function $g\in L_\infty(K)$ with $||g||_\infty = \sup\{||g_j||:j\in \mathcal{J}\} \leq (1+\epsilon)||f||_\infty$ such that $f = g\circ T - g$, which completes the proof of Theorem \ref{TheoremNowhereConstant}.\\

\section{Kwapie\'{n} Theorem for elementary functions} \label{Section:ProofCountablyValuedFunctions}
In this section we prove the Theorem \ref{Main result} for mean zero functions taking only countably many values. More precisely, we establish the following result.
\begin{Thm}\label{TheoremCountablyValued} Let $K\subseteq [0,1]$ be measurable. Let $f\in L_\infty(K)$ be a mean zero real-valued function taking at most countably many values. Then there exists some measure preserving transformation mod0 $T$ of $K$ and a function $g\in L_\infty(K)$ with $||g||_\infty\leq ||f||_\infty$, such that $f = g\circ T - g$.	 
\end{Thm}
\begin{proof}
	We will prove this theorem by considering several cases, increasing the level of generality
	
	1. Let $a\in[0,1]$ and $f_a=(1-a)\chi_{[0,a)}-a\chi_{[a,1)}$. Clearly, $\int_0^1 f_a d\lambda=0$ and $\|f_a\|_\infty=\max(a,1-a)\geq 1/2$.
	Set $g(t)=t-1/2,\ T(t)=\{t-a\}$ (by $\{t\}$ we denote the fractional part of $t$, that is the distance from $t$ to the closest integer which does not exceed $t$). Then $T$ is measure preserving, $g(T(t))-g(t)=f_a(t)$ for all $t\in[0,1)$ and $\|g\|_\infty=1/2\leq \|f_a\|_\infty$.
	
	2. Let now $a\in[0,1]$, $f=\alpha\chi_{[0,a)}+\beta\chi_{[a,1)}$ and $\int_0^1 f d\lambda=0$. Then $\alpha a+\beta(1-a)=0$ and $f=\frac{\alpha}{1-a}f_a$. Therefore, this case can be reduced to the preceding.
	
	3. Let $a,b\in[0,1],\ a<b$, $f=\alpha\chi_{[0,a)}+\beta\chi_{[a,b]}$ and $\int_0^1 f d\lambda=0$. This case can be reduced to the preceding as follows.
	We define $\widetilde{f}(t) = f(bt)$
	so that $\widetilde{f} = \alpha \chi_{[0,\frac{a}{b})} + \beta\chi_{[\frac{a}{b},1]}$ and $\int_{0}^1\widetilde{f}d\lambda =0$.
	Hence we find $\widetilde{g}$ with $||\widetilde{g}||_\infty\leq ||\widetilde{f}||_\infty$ and measure preserving transformation $\widetilde{T}$ of $[0,1]$ such that $\widetilde{f} = \widetilde{g}\circ \widetilde{T} -\widetilde{g}$. Now define $g(t) = \widetilde{g}(\frac{t}{b})$ and $T(t) = b\widetilde{T}(\frac{t}{b})$ for $t\leq b$ and $g(t) = 0$ and $T(t) = t$ for $t>b$. We then find for $t\leq b$ that $f(t) = \widetilde{f}(\frac{t}{b}) = \widetilde{g}(\widetilde{T}(\frac{t}{b})) - \widetilde{g}(\frac{t}{b}) = g(T(t)) - g(t)$ and for $t>b$ we find $f(t) = 0 =g(T(t)) - g(t)$. Moreover, we have $||g||_\infty = ||\widetilde{g}||_\infty \leq ||\widetilde{f}||_\infty = ||f||_\infty$
	
	4. Let $A$, $B$ be disjoint measurable sets. Let $f = \alpha\chi_A +\beta\chi_{B} \in L_\infty[0,1]$ be mean zero.  We set $C = [0,1]\setminus (A\cup B)$. By Theorem \ref{Prelim:isomorphism} there exists a measure preserving transformation mod0 $S$ of $[0,1]$ such that 
$$
S(A) = [0,\lambda(A)),\quad S(B) = [\lambda(A),\lambda(A) + \lambda(B)],\quad S(C) = (\lambda(A) + \lambda(B),1].
$$ 
Letting $\widetilde{f} = f\circ S^{-1}$, we obtain  
$$\widetilde{f} = \alpha\chi_{[0,\lambda(A))} + \beta\chi_{[\lambda(A),\lambda(A)+\lambda(B)]}.
$$ 
Hence, appealing to the case, we find a function $\widetilde{g}\in L_\infty[0,1]$ with $\|\widetilde{g}\|_\infty \leq \|\widetilde{f}\|_\infty$ and a measure preserving transformation $\widetilde{T}$ of $[0,1]$ such that $\widetilde{f} = \widetilde{g}\circ\widetilde{T} - \widetilde{g}$. Furthermore, $\widetilde{T}$ is the identity on $[\lambda(A)+\lambda(B),1]$.
	Now define $T = S^{-1}\circ \widetilde{T} \circ S$ which is a measure preserving transformation mod0 of $[0,1]$ and define $g = \widetilde{g}\circ S\in L_\infty[0,1]$. We have $f = \widetilde{f}\circ S = g\circ T - g$.
	Moreover we have $\|g\|_\infty = \|\widetilde{g}\|_\infty \leq \|\widetilde{f}\|_\infty = \|f\|_\infty$.
Further noting that $T$ is the identity on $[0,1]\setminus (A\cup B)$ we can also consider $T$ as a measure preserving transformation mod0 of $A\cup B$. Hence $f|_{A\cup B} = g|_{A\cup B}\circ T|_{A\cup B} - g|_{A\cup B}$.
	
	5. Let $K\subseteq [0,1]$ be measurable. Let $f\in L_\infty(K)$ be mean zero and taking at most countably many values. Without loss of generality, we write $f = \sum_{i=1}^{\infty}\alpha_i\chi_{A_i}$ for some scalars $\alpha_i\in \RR\setminus\{0\}$ and some pairwise disjoint measurable sets $A_i\subseteq K$.
	
Let $\{B_j^+,B_j^-\}_{j\in \mathcal{J}}$ be a collection of pairs with $\mathcal{J}$ being an index set, such that all sets $B_j^+,B_j^-$ are disjoint, of positive measure and such that for $j\in \mathcal{J}$ there exists $i_1,i_2\in \NN$ with $B_j^+\subseteq A_{i_1}$ and $B_j^-\subseteq A_{i_2}$ and such that $\alpha_{i_1}\lambda(B_j^+) + \alpha_{i_2}\lambda(B_j^-) = 0$. Now, consider the set of all such collections equipped with the partial ordering $\subseteq$ given by inclusion.
	Now, suppose we have some chain $\{\{B_j^+,B_j^-\}_{j\in \mathcal{J}_i}\}_{i\in \mathcal{I}}$, where $\mathcal{I}$ is some index set, then if we set $\mathcal{J} = \bigcup_{i\in \mathcal{I}} \mathcal{J}_i$ we find a upper bound $\{B_j^+,B_j^-\}_{j\in \mathcal{J}}$ for the chain. 
An appeal to Zorn's Lemma yields a maximal element in the set of collections,
	$\{B_j^+,B_j^-\}_{j\in \mathcal{J}}$ for some countable set $\mathcal{J}$.
Let us set 
$$Z = \bigcup_{j\in\mathcal{J}}(B_j^+\cup B_j^-)
$$ and suppose that 
$$\lambda(supp f\setminus Z)\not=0.
$$ 
Taking into account that 
$\int_{B_j^+\cup B_j^-}fd\lambda = 0$ for all $j\in \mathcal{J}$, we infer that $\int_{supp f\setminus Z}fd\lambda = 0$ and, hence, we can select 
$i_1,i_2\in \NN$ and sets $B^+\subseteq A_{i_1}\cap (supp f\setminus Z)$ and $B^-\subseteq A_{i_2}\cap (supp f\setminus Z)$ of positive measure such that $B^+,B^-$ are disjoint with all sets $B_j^+, B_j^-$ for $j\in \mathcal{J}$, and such that $\alpha_{i_1}> 0 > \alpha_{i_2}$.
	This means that we can find $B'^+\subseteq B^+$ and $B'^-\subseteq B^-$ of positive measure such that $\alpha_{i_1}\lambda(B'^+) + \alpha_{i_2}\lambda(B'^-)=\int_{B'^+\cup B'^-}fd\lambda =0$. Now this means that the selection$\{B_j^+,B_j^-\}_{j\in \mathcal{J}}$  is not maximal, which is a contradiction. Hence we conclude that 
$$\lambda(supp f\setminus Z) = 0.
$$
Thus, for each $j\in \mathcal{J}$ we can, by referring to the preceding case, find a $g_j\in L_\infty(B_j^+\cup B_j^-)$ with $\|g_j\|_\infty \leq \|f\|_\infty$ and a measure preserving transformation mod0 $T_j$ of $B_j^+\cup B_j^-$ such that $f|_{B_j^+\cup B_j^-} = g_j\circ T_j -g_j$ on $B_j^+\cup B_j^-$. Hence, defining $T$ and $g$ as  $T_j$ and $g_j$ respectively on $B_j^+\cup B_j^-$ for $j\in \mathcal{J}$ and setting $T(x) = x$ and $g(x)=0$ for $x$ in the null set $K\setminus Z$ yields the measure preserving transformation mod0 $T$ of $K$ and the function $g\in L_\infty(K)$ satisfying
	$f = g\circ T - g$. Moreover $\|g\|_\infty\leq \sup_{j\in \mathcal{J}}\|g_j\|_\infty \leq \|f\|_\infty$. The proof of Theorem \ref{TheoremCountablyValued} is completed.
\end{proof}

\section{Completing the proof of Theorem \ref{Main result}}
\label{Section:ProofFullTheorem}
We are now in a position to complete the proof of  Kwapie\'{n}'s Theorem \ref{Main result}. For convenience, we restate it below.
\begin{Thm}
	Let $f\in L_\infty[0,1]$ be a real-valued mean zero function. For any $\epsilon>0$ there exists a measure preserving transformation mod0 $T$ of $[0,1]$ and a function $g\in L_\infty[0,1]$ with
	$\|g\|_\infty \leq (1+\epsilon)\|f\|_\infty$ so that 
	$f = g\circ T - g$.
	\end{Thm}
	\begin{proof}
		Let $f\in L_\infty[0,1]$ be mean zero. We will partition the interval $[0,1]$ into certain subsets on which $f$ is mean zero. We do this as follows.
		We let 
$$D' = \{y\in \RR: \lambda(f^{-1}(\{y\})) >0 \}$$ and set 
$D = f^{-1}(D')$. The function $f$ takes only countably many values on $D$, since every value on $D$ is taken on a set of positive measure. We set 
$$D^\pm = \{f^\pm \geq 0\}\cap D.
$$
We assume $\int_{D}fd\lambda \geq 0$, the case that $\int_{D}fd\lambda <0$ then follows by considering $-f$. 
We now set 
$$
C = (D^+\cap [0,R]) \cup D^-
$$ 
for some 
$0\leq R\leq 1$ such that $\int_{C}fd\lambda = 0$.
Further, we consider the sets 

$$
C' = [0,1]\setminus C,\quad C_1 = C'\cap D, 
\quad C_2 = C'\setminus D.
$$ 

As $C_1\subseteq D\setminus D^-$ we have 
$f|_{C_1}>0$ and hence $\int_{C_1}fd\lambda \geq 0$. Further, as $C_2\subseteq [0,1]\setminus D$ we have 
$\lambda(f|_{C_2}^{-1}(\{y\})) = 0$ for all $y\in \RR$.
As $f$ is mean zero on $[0,1]$ and on $C$, we have
$$\int_{C_1}fd\lambda + \int_{C_2}fd\lambda = \int_{C'}fd\lambda = -\int_{C}fd\lambda = 0.
$$
Hence, $\int_{C_2}fd\lambda\leq 0$.
We further denote 
$$
C_2^\pm = \{f^\pm \geq 0\} \cap C_2
$$
and define 
$$B_0 = C_2^+ \cup (C_2^- \cap [0,R])$$ for some $0\leq R\leq 1$ so that 
$$\int_{B_0}fd\lambda = 0.$$

Now, at last, we let 
 $$\widetilde{C}_2 = C_2\setminus B_0.$$ 
 As $\widetilde{C}_2\subseteq C_2^-\setminus C_2^+$ we have that $f|_{\widetilde{C}_2} <0$.
		We further have 
		$$\int_{\widetilde{C}_2}fd\lambda = \int_{C_2}fd\lambda = -\int_{C_1}fd\lambda.$$
		
We let $(y_i)_{i\geq 1}$ be an (either finite or infinite) enumeration of $f(C_1)$ and we set 
$$A_i = f^{-1}(\{y_i\})\cap C_1.
$$ 
Further, we let $r_0 = 0$. 
Now, as 
$$\int_{C_1}fd\lambda + \int_{\widetilde{C}_2}fd\lambda = 0
$$ 
we can recursively choose $r_i$ for $i\geq 1$ such that 
$r_i\geq r_{i-1}$ and 
$$\int_{[r_{i-1},r_{i}]\cap \widetilde{C}_2}fd\lambda + \int_{A_i}fd\lambda = 0.
$$
We then set 
$$B_i = [r_{i-1},r_{i}]\cap \widetilde{C}_2.
$$
We have now partitioned $[0,1]$ into the sets $C, B_0,$ and  $A_i\cup B_i$ for $i\geq 1$ (or, for possibly finitely many $i$). On each of these sets $f$ is mean zero. Now, as $f$ takes only countably many values on $C$, we can use Theorem \ref{TheoremCountablyValued} to get a measure preserving transformation mod0 $T_C$ of $C$ and a function $g_C\in L_\infty(C)$ with $\|g_C\|_\infty\leq \|f||_\infty$ so that $f|_C = g_C\circ T_C - g_C$.

On $E_0:=B_0$ we have that $\lambda(f|_{E_0}^{-1}(\{y\})) = 0$ for all $y\in \RR$, hence we can use Theorem \ref{TheoremNowhereConstant}  to obtain a measure preserving transformation mod0, $T_{E_0}$ of $E_0$, and a function $g_{E_0}\in
		L_\infty(E_0)$ with $||g_{E_0}||_\infty\leq (1+\epsilon)||f||_\infty$ so that $f|_{E_0} = g_{E_0}\circ T_{E_0} - g_{E_0}$. Last, we use Theorem \ref{TheoremNowhereConstant} on $E_i := A_i\cup B_i$ for $i\geq 1$ with $\kappa = y_i$ to obtain a measure preserving transformation mod0 $T_{E_i}$ of $E_i$ and a function $g_{E_i}\in L_\infty(E_i)$ with $||g_{E_i}||_\infty\leq (1+\epsilon)||f||_\infty$ so that $f|_{E_i} = g_{E_i}\circ T_{E_i} - g_{E_i}$.
		
Finally, we define the measure preserving transformation mod0 $T$ of $[0,1]$  by setting $T|_C = T_C$ and $T|_{E_i} = T_{E_i}$
		for $i\geq 0$, and on the remaining null set we define $T$ as the identity. Likewise, we define the function $g\in L_\infty[0,1]$ by setting $g|_{C} = g_C$, and $g|_{E_i} = g_{E_i}$ for $i\geq 0$. We then have $f = g\circ T -g$ as well as  the bound $\|g\|_\infty\leq \sup\{\|g_C\|_\infty\}\cup \{\|g_{E_i}\|_\infty: i\geq 1\}\leq (1+\epsilon)\|f\|_\infty$. This completes the proof.
	\end{proof}

\section*{Acknowledgements}
Some part of this research was done by the second named author (M.B.) for his Bachelor of Science thesis. He thanks his supervisor Mark Veraar for giving him helpful insights and feedback during the research. The third named author (F.S.) thanks Professors Kwapie\'{n} and Bogachev for useful discussions concerning \cite{Kwapien}. \\

Department of Mathematics, \\National University of Uzbekistan,\\
	Vuzgorodok, 100174,\\ Tashkent, Uzbekistan\\
\email{ber@ucd.uz}\\

Delft Institute of Applied Mathematics, \\Delft University of Technology,\\
	P.O. Box 5031,
	2600 GA,\\ Delft, The Netherlands\\
\email{m.j.borst@student.tudelft.nl}\\
\email{m.j.borst@outlook.com}\\

School of Mathematics and Statistics,\\ University of New South Wales,\\ Kensington, NSW 2052, Australia\\
\email{f.sukochev@unsw.edu.au}\\

\end{document}